\documentclass[psamsfonts,12pts]{amsart}

\usepackage{mathpazo}
\usepackage[utf8]{inputenc}
\usepackage[margin=2cm,centering]{geometry}
\usepackage{amsmath}
\usepackage{amsthm}

\usepackage{amssymb}
\usepackage{amscd}
\usepackage{amsfonts}
\usepackage{amsbsy}
\usepackage{graphicx}
\usepackage[dvips]{psfrag}
\usepackage{array}
\usepackage{color}
\usepackage{epsfig}
\usepackage{url}
\usepackage{overpic}
\usepackage{epstopdf}
\usepackage[normalem]{ulem}
\usepackage{tikz}

\usepackage{caption}
\usepackage{subcaption}
\usepackage{multirow}
\usepackage{indentfirst}
\usepackage{mathrsfs}
\usepackage{hyperref}
\usepackage{placeins}
\usepackage{flafter}
\usepackage{tcolorbox}
\usepackage{amsmath}
\usetikzlibrary{arrows}

\usepackage[foot]{amsaddr}

\newtheorem {theorem} {Theorem}
\newtheorem {definition} {Definition}

\newtheorem {lemma}  [theorem]{Lemma}

\newtheorem {remark} [theorem]{Remark}

\allowdisplaybreaks

\begin{document}
\renewcommand{\arraystretch}{1.5}

\title[Sliding trajectories are closed]{A proof that all the sliding trajectories of generic inelastic piecewise linear dynamical systems over the torus are closed}
\author[M. D. A. Caldas, R. M. Martins]
{Mayara D. A. Caldas$^{1}$, Ricardo M. Martins$^{2}$}
	
\address{$^{1}$ Universidade Federal do Rio de Janeiro, Campus Duque de Caxias Professor Geraldo Cidade, Duque de Caxias, RJ, Brasil.
} 

\address{$^{2}$ Instituto de Matemática, Estatística e Computação Científica, Universidade Estadual de Campinas, Rua S\'{e}rgio Buarque de Holanda, 651, Cidade Universit\'{a}ria Zeferino Vaz, 13083--859, Campinas, SP, Brazil.}

\email{$^{1}$ mcaldas@ime.unicamp.br}
\email{$^{2}$ rmiranda@unicamp.br}
	
	
	
\maketitle

\begin{abstract}
In this article, we consider piecewise smooth differential equations $Z_{X_-X_+}$, where $X_-$ and $X_+$ are linear vector fields in dimension 3, having the torus as discontinuity manifold. We consider that $Z_{X_-X_+}$ is an inelastic vector field over the torus. We classify the set of tangency points on the torus for certain conditions and describe the behavior of the trajectories of the sliding vector field over the torus. We prove that, under generic conditions, all the trajectories of a piecewise smooth linear inelastic vector field over the torus are closed. We also provide a result about the topological equivalence of inelastic vector fields over the torus.
\end{abstract}

\section{Introduction and statement of main results}

One of the main goals in the study of discontinuous dynamical systems is to understand the qualitative behavior of their solutions. The mathematical formalization of these systems was initiated by Filippov (see \cite{f1}), and important results were obtained by Teixeira, Sotomayor, and others in the 1980s and 1990s (see \cite{marco2,marco4}). Furthermore, the theory of generic bifurcations was well structured by Teixeira, Guardia, and Seara (see \cite{GST,survey}). A fundamental part of these studies relies on previous work concerning vector fields defined on manifolds with boundary (see \cite{v1,v2}).

These systems naturally arise in many applications, such as control theory, mechanical systems with impacts, and electronic circuits, where discontinuities result from switching rules or abrupt changes in dynamics (see, for instance, \cite{ap1,ap2,ap3,ap4}).

A fundamental concept in this context is structural stability (see \cite{GST}), which establishes conditions under which the qualitative behavior of a system remains unchanged under small perturbations. Structural stability provides a framework for classifying systems and understanding which features of their dynamics are robust and which are sensitive to changes in parameters or external influences.

In recent years, particular attention has been devoted to the class of inelastic discontinuous dynamical systems, characterized by a discontinuity manifold that consists only of sliding and escaping regions. These systems exhibit complex dynamics due to the interaction between vector fields on either side of the discontinuity and the resulting sliding motion. The absence of crossing regions introduces peculiar phenomena that differ substantially from those observed in classical Filippov systems. For a detailed discussion of these aspects, see \cite{marco3}, \cite{Bro}, \cite{VO}, and \cite{villanueva2022global}. Understanding the stability and bifurcations of such systems remains an active area of research, with important implications for modeling real-world phenomena (see, for instance, \cite{ap5}).

Another central topic in the qualitative theory of dynamical systems is the study of invariant elements, particularly invariant tori, which are fundamental objects in Hamiltonian dynamics and in systems exhibiting quasi-periodic motions. Piecewise-smooth systems with invariant tori have been studied in \cite{r1,r2,t1}, addressing both theoretical aspects and applications.

In this work, we address the question of piecewise-smooth systems that admit an invariant torus composed solely of sliding motion and fold points. In the next subsection, we describe the main results obtained.

\subsection{Main results} Let $U\subset\mathbb R^3$ be an open bounded connected set with $(0,0,0)\in U$. Consider $T^2$ a torus embedded in $U$; to fix coordinates, consider the canonical torus in $\mathbb R^3$ with major radius $2$ and minor radius $1$ such that $T^2\subset U$. In other words, consider $h:U\subset\mathbb R^3\rightarrow\mathbb R$ given by $h(x,y,z)=(x^2+y^2+z^2+3)^2-16(x^2+y^2)$, so that $T^2=h^{-1}(\{0\})$. We call $h^{-1}((-\infty,0))$ the interior of the torus and $h^{-1}((0,\infty))$ the exterior of the torus. Note that the interior of the torus is a compact region.

Let $C^k(U, \mathbb{R}^3)$ be the set of vector fields of class $C^r$ defined in $U$. Given $X_-,X_+\in C^k(U, \mathbb{R}^3)$, consider the piecewise smooth vector field
\[Z_{X_-X_+}(x,y,y)=\left\{
\begin{array}{lcl}
	X_-(x,y,z),& h(x,y,z)< 0,\\
	X_+(x,y,z),& h(x,y,z)> 0,
\end{array}
\right.\]
where $X_-$ is defined in the region interior to $T^2$ and $X_+$ in the region exterior to $T^2$. We denote by $\mathfrak{Z}^r$  the set of vector fields of type $Z_{X_-X_+}$, which can be taken as $\mathfrak{Z}^r=C^k(U, \mathbb{R}^3)\times C^k(U, \mathbb{R}^3)$ inheriting a topology from this product.

We say that the piecewise smooth vector field $Z_{X_-X_+}$ is inelastic over $T^2$ if 
\begin{equation}\label{inelastic}
	X_+h(x,y,z)=-X_-h(x,y,z)
\end{equation} 
for all $(x,y,z)\in T^2$, thus $\left(X_+h+X_-h)\right|_{T^2}=0$. We are going to denote by $\mathfrak{Z}^{l}$ the set of piecewise smooth differential equations $Z_{X_-X_+}$ where $X_+$ and $X_-$ are linear vector fields with the origin being the singularity, and $\mathfrak{Z}^{l}_{\mathcal{I}}$ a subset of $\mathfrak{Z}^{l}$ where $Z_{X_-X_+}$ is inelastic over $T^2$. 

Our main results are the following.

\begin{theorem}\label{propfinal1}Under generic conditions, if $Z_{X_-X_+}\in\mathfrak{Z}^{l}_{\mathcal{I}}$, then the trajectories of $Z_{X_-X_+}$ over the torus are closed curves. 
\end{theorem}


\begin{theorem}\label{44r4r4r4}There is a set $\mathfrak{z}\subset\mathfrak{Z}^{l}_{\mathcal{I}}$ such that if $Z_{X_-X_+}\in \mathfrak{z}$ and $\tilde Z_{\tilde{X}_-\tilde{X}_+}\in \mathfrak{z}$ is a vector field in a small neighborhood of $Z_{X_-X_+}\in \mathfrak{z}$, then the sliding vector fields $Z^{s}$ and $\tilde Z^{s}$, associated to $Z_{X_-X_+}$ and $\tilde Z_{\tilde{X}_-\tilde{X}_+}$, respectively, and defined in the torus, are topologically equivalent.
\end{theorem}

This paper is divided as follows. Section \ref{fili} presents the basic definitions about piecewise smooth vector fields according to Filippov’s convention \cite{GST}. In Section \ref{secinelastic} we prove Theorem \ref{propfinal1}. In Section \ref{sectang} we classify of the tangency sets on the torus. Section \ref{proof2} is dedicated to proof of Theorem \ref{44r4r4r4}. In Section 6 we present some conclusions.

\section{Filippov convention for piecewise smooth differential equations}
\label{fili}
Let $X_+$ and $X_-$ be smooth vector fields defined in an open and convex subset $U\subset \mathbb{R}^3$ and, without loss of generality, assume that the origin belongs to $U$. Consider $f:U\rightarrow\mathbb{R}$ a function $\mathcal{C}^r$, with $r>1$ ($\mathcal{C}^r$ denotes the set of continuously differentiable functions of order $r$), for which $0$ is a regular value. Thus, the curve  $\Sigma=f^{-1}(0)\cap U$ is a submanifold of dimension 1 and divides the open set $U$ into two open sets,
$$\Sigma^{+}=\{p\in U | f(p)>0\} \quad \mbox{and} \quad \Sigma^{-}=\{p\in U | f(p)<0\}.$$

A Filippov system is a piecewise smooth vector field defined in the following form: 
\begin{equation}\label{sist.fil}
	Z_{X_-X_+}(p)=\begin{cases}
        X_-(p), & p \in \Sigma^{-},\\
		X_+(p),& p \in \Sigma^{+}, 
	\end{cases}
\end{equation}
in order to identify the components of the field. Moreover, we assume that $X_+$ and $X_-$ are fields of class $\mathcal{C}^k$ with $k>1$ in $\overline{\Sigma^{+}}$ and $\overline{\Sigma^{-}}$, respectively, where $\overline{\Sigma^{\pm}}$ denotes the closure of $\Sigma^{\pm}$.    

We denote by $\Omega^k(U)$ the space of vector fields of the type (\ref{sist.fil}), that can be taken as $\Omega^k(U)=C^k(U, \mathbb{R}^3)\times C^k(U, \mathbb{R}^3)$, where, by an abuse of notation, we denote by $C^k(U, \mathbb{R}^3)$ both the set of vector fields of class $\mathcal{C}^k$ defined in $\overline{\Sigma^{+}}$ and $\overline{\Sigma^{-}}$. We consider $\Omega^k(U)$ endowed with the product topology.

To establish the dynamic given by a Filippov vector field $Z_{X_-X_+}$ in $U$, we need to define the local trajectory through a point $p\in U$, that is, we must define the flow $\varphi_z(t,p)$ of (\ref{sist.fil}). If $p \in \Sigma_{}^{\pm}$, then the trajectory through $p$ is given by the fields $X_+$ and $X_-$ in the usual way. However, if $p \in \Sigma$, we must be more careful defining the trajectory. In order to extend the definition of trajectory for $\Sigma$, we are going to divide the discontinuity submanifold $\Sigma$ in the closure of three disjoint regions: 
\begin{itemize}
	\item Crossing region: $ \Sigma^{c}=\{p \in \Sigma |  X_+f(p)\cdot X_-f(p) >0\}$,
	\item Sliding region: $ \Sigma^{s}=\{p \in \Sigma |  X_+f(p) <0,  X_-f(p) >0\}$,
	\item Escape region: $ \Sigma^{e}=\{p \in \Sigma |  X_+f(p) >0,  X_-f(p) <0\}$,
\end{itemize}
where $X_+f(p)=\langle X_+(p), \nabla f(p)\rangle$ and $X_-f(p)=\langle X_-(p), \nabla f(p)\rangle$ are Lie's derivative of $f$ with respect to the field $X_+$ in $p$ and $f$ with respect to the field $X_-$ in $p$, respectively. These three regions are open subsets of $\Sigma$ in the induced topology and can have more than one convex component. 

We can observe that when defining the regions above we aren't including the {tangent points}, that is, the points $p \in \Sigma$ for which $X_+f(p)=0$ or $X_-f(p)=0$. These points are in the boundaries of the regions $\Sigma^{c}$, $\Sigma^{s}$ and $\Sigma^{e}$, which are going to be denoted by $\partial \Sigma^{c}$, $\partial \Sigma^{s}$ and $\partial\Sigma{}^{e}$, respectively.

Note that if $X_+(p)=0$, then $X_+f(p)=0$, so the critical points of $X_+$ in $\Sigma$ are also included in the tangent points. Now, if $X_+(p)\neq 0$ and $X_+f(p)=0$, we confirm that the trajectory of $X_+$ passing through $p$ is, indeed, tangent to $\Sigma$.

We can distinguish the tangency types between a smooth field and a manifold depending on how the contact between the trajectory of the field and the manifold occurs. Next, we define two types of tangency.

\begin{definition}
	A smooth vector field $X_+$ admits a fold or quadratic tangency with $\Sigma$ in a point $p\in \Sigma$ if $X_+f(p)=0$ and $X_+^2f(p)\neq 0$, being $X_+^2f(p)=\langle X_+(p), \nabla X_+f(p)\rangle$.
\end{definition}

\begin{definition}
	A smooth vector field $X_+$ admits a cusp or cubic tangency with $\Sigma$ in a point $p\in \Sigma$ if $X_+f(p)=X_+^2f(p)=0$ and $X_+^3f(p)\neq 0$, being $X_+^3f(p)=\langle X_+(p), \nabla X_+^2f(p)\rangle$.
\end{definition}

Let's define the trajectory passing through a point $p$ in $\Sigma^{c}$, $\Sigma^{s}$ and $\Sigma^{e}$. If $p \in \Sigma^{c}$, both the vector fields $X_+$ and $X_-$ point to $\Sigma^{+}$ or $\Sigma^{-}$ and, therefore, it is sufficient to concatenate the trajectories of $X_+$ and $X_-$ that pass through $p$. If $p\in \Sigma^{s}\cup \Sigma^{e}$, we have that the vector fields point to opposite directions, thus, we can't concatenate the trajectories. In this way, the local orbit is provided by the Filippov convention. Thus, we define the { sliding vector field}
\begin{equation}\label{campo_deslizante}
	Z^{s}(p)=\frac{1}{X_-f(p)-X_+f(p)}(X_-f(p)X_+(p)-X_+f(p)X_-(p)).
\end{equation}  

Note that $Z^s$ represents the convex linear combination of $X_+(p)$ and $X_-(p)$, so that $Z^{s}$ is tangent to $\Sigma$, moreover, its trajectories are contained in $\Sigma^{s}$ or $\Sigma^{e}$. Thus, the trajectory through $p$ is the trajectory defined by the sliding vector field in (\ref{campo_deslizante}).

We present now the definition of topological equivalence and structural stability for Filippov systems. We indicate the reference \cite{marco2} 
for more details.

\begin{definition}
	Consider $Z_{X_-,X_+}, \tilde{Z}_{\tilde{X}_-,\tilde{X}_+} \in \Omega^k(U)$ a Filippov system and $p\in U$. We say that $Z_{X_-,X_+}$ and $\tilde{Z}_{\tilde{X}_-,\tilde{X}_+}$ are topologically equivalent at $p$ if there exists neighborhoods $V$ and $\tilde{V}$ of $p$ in $\mathbb{R}^n$ and a homeomorphism $\phi:V \rightarrow \tilde{V}$ which keeps $\Sigma$ invariant and takes trajectories of $Z_{X_-,X_+}$ into trajectories of $\tilde{Z}_{\tilde{X}_-,\tilde{X}_+}$ preserving time orientation. 
\end{definition}

From this definition we can define the concept of local structural stability.

\begin{definition}
	We say that a Filippov system $Z_{X_-,X_+}\in \Omega^k(U)$ is structurally stable at $ p\in U$ if there exists an open neighborhood $W$ of $Z_{X_-,X_+}$ in $\Omega^k(U)$ such that, if $\tilde{Z}_{\tilde{X}_-,\tilde{X}_+}\in W$ is a Filippov system, then $\tilde{Z}_{\tilde{X}_-,\tilde{X}_+}$ and $Z_{X_-,X_+}$ are topologically equivalent at $p$. 
\end{definition}

Both definitions can be given globally, as follows.

\begin{definition}
	Consider $Z_{X_-,X_+}, \tilde{Z}_{\tilde{X}_-,\tilde{X}_+} \in \Omega^k(U)$ Filippov systems. We say that they are topologically equivalent if there exists a homeomorphism $\phi:U \rightarrow \tilde{U}$ that keeps $\Sigma$ invariant and takes orbits of $Z_{X_-,X_+}$ into orbits of $\tilde{Z}_{\tilde{X}_-,\tilde{X}_+}$ preserving time orientation. 
\end{definition}

\begin{definition}
	We say that a Filippov system $Z_{X_-,X_+}\in \Omega^k(U)$ is structurally stable if there exists an open neighborhood $W$ of $Z_{X_-,X_+}$ in $\Omega^k(U)$ such that, if $\tilde{Z}_{\tilde{X}_-,\tilde{X}_+}\in W$ is a Filippov system, then $\tilde{Z}_{\tilde{X}_-,\tilde{X}_+}$ and $Z_{X_-,X_+}$ are topologically equivalent. 
\end{definition}

\section{Review and classification of inelastic systems}\label{secinelastic}

In this section, we present some properties of the piecewise smooth vector field $Z_{X_-X_+}$, when it is inelastic over $T^2$ and $X_-, X_+$ are linear vector fields. Moreover, we prove Theorem \ref{propfinal1}.

\begin{lemma}\label{lema1}Let $X_+$ and $X_-$ be linear vector fields in $\mathbb{R}^3$. Consider 
	$$X_+(x,y,z)=A(x,y,z)^{T}\quad \mbox{and}\quad X_-(x,y,z)=B(x,y,z)^{T},$$ for $A=(a_{i,j})$ and $B=(b_{i,j})$ matrices. If $Z_{X_-X_+}$ is inelastic over the torus $T^2$, then
	\[
	B= \left( \begin {array}{ccc} -a_{{1,1}}&-a_{{1,2}}-a_{{2,1}}-b_{{2,1}}&
	-a_{{1,3}}\\ \noalign{\medskip}b_{{2,1}}&-a_{{2,2}}&-a_{{2,3}}
	\\ \noalign{\medskip}-a_{{3,1}}&-a_{{3,2}}&-a_{{3,3}}\end {array}
	\right).
	\]	
\end{lemma}

\begin{proof}We have 
	\[X_+(x,y,z)= \left(\begin{array}{ccc}
		a_{1,1} & a_{1,2} & a_{1,3}\\ \noalign{\medskip}
		a_{2,1} & a_{2,2} & a_{2,3}\\ \noalign{\medskip}
		a_{3,1} & a_{3,2} & a_{3,3} 
	\end{array}\right) 
	\left(\begin{array}{c}
		x\\ \noalign{\medskip}
		y\\ \noalign{\medskip}
		z
	\end{array}\right)=
	\left(\begin{array}{c}
		a_{1,1}x+a_{1,2}y+a_{1,3}z\\ \noalign{\medskip}
		a_{2,1}x+a_{2,2}y+a_{2,3}z\\ \noalign{\medskip}
		a_{3,1}x+a_{3,2}y+a_{3,3}z
	\end{array}\right) \]
	and
	\[X_-(x,y,z)= \left(\begin{array}{ccc}
		b_{1,1} & b_{1,2} & b_{1,3}\\ \noalign{\medskip}
		b_{2,1} & b_{2,2} & b_{2,3}\\ \noalign{\medskip}
		b_{3,1} & b_{3,2} & b_{3,3} 
	\end{array}\right) 
	\left(\begin{array}{c}
		x\\ \noalign{\medskip}
		y\\ \noalign{\medskip}
		z
	\end{array}\right)=
	\left(\begin{array}{c}
		b_{1,1}x+b_{1,2}y+b_{1,3}z\\ \noalign{\medskip}
		b_{2,1}x+b_{2,2}y+b_{2,3}z\\ \noalign{\medskip}
		b_{3,1}x+b_{3,2}y+b_{3,3}z
	\end{array}\right). \]
	
	Considering $h(x,y,z)=(x^2+y^2+z^2+3)^2-16(x^2+y^2)$, we have
	$$\nabla h(x,y,z) =(4x(x^2+y^2+z^2+3)-32x,4y(x^2+y^2+z^2+3)-32y,4z(x^2+y^2+z^2+3)).$$
	Thus, 
	\begin{align*}
		X_+h(x,y,z)&=
		4(x^2+y^2+z^2+3)[a_{1,1}x^2+a_{2,2}y^2+a_{3,3}z^2+xy(a_{1,2}+a_{2,1})\\ \noalign{\medskip}
		&\quad+xz(a_{1,3}+a_{3,1})+yz(a_{2,3}+a_{3,2})]-32[a_{1,1}x^2+a_{2,2}y^2\\ \noalign{\medskip}
		&\quad+xy(a_{1,2}+a_{2,1})+a_{1,3}xz+a_{2,3}yz]
	\end{align*}
	and
	\begin{align*}
		X_-h(x,y,z)
		&=4(x^2+y^2+z^2+3)[b_{1,1}x^2+b_{2,2}y^2+b_{3,3}z^2+xy(b_{1,2}+b_{2,1})\\ \noalign{\medskip}
		&\quad+xz(b_{1,3}+b_{3,1})+yz(b_{2,3}+b_{3,2})]-32[b_{1,1}x^2+b_{2,2}y^2\\ \noalign{\medskip}
		&\quad +xy(b_{1,2}+b_{2,1})+b_{1,3}xz+b_{2,3}yz].
	\end{align*}
	
	By hypothesis, $X_+$ and $X_-$ are inelastic, thus by equation (\ref{inelastic}), we have 
	$$X_+h(x,y,z)=-X_-h(x,y,z),$$ 
	that is,
	\begin{align*}
		&\quad4(x^2+y^2+z^2+3)[a_{1,1}x^2+a_{2,2}y^2+a_{3,3}z^2+xy(a_{1,2}+a_{2,1})+xz(a_{1,3}+a_{3,1})\\ \noalign{\medskip}
		&\quad+yz(a_{2,3}+a_{3,2})]-32[a_{1,1}x^2+a_{2,2}y^2+xy(a_{1,2}+a_{2,1})+a_{1,3}xz+a_{2,3}yz]\\ \noalign{\medskip}
		&=-4(x^2+y^2+z^2+3)[b_{1,1}x^2+b_{2,2}y^2+b_{3,3}z^2+xy(b_{1,2}+b_{2,1})+xz(b_{1,3}+b_{3,1})\\ \noalign{\medskip}
		&\quad+yz(b_{2,3}+b_{3,2})]+32[b_{1,1}x^2+b_{2,2}y^2+xy(b_{1,2}+b_{2,1})+b_{1,3}xz+b_{2,3}yz]. 
	\end{align*}
	In this way, we obtain
	$$b_{1,1}=-a_{1,1}, \quad b_{2.2}=-a_{2,2}\quad b_{3,3}=-a_{3,3} \quad b_{1,2}+b_{2,1}=-a_{1,2}-a_{2,1},$$
	$$b_{1,3}+b_{3,1}=-a_{1,3}-a_{3,1} \quad b_{2,3}+b_{3,2}=-a_{2,3}-a_{3,2}, \quad b_{1,3}=-a_{1,3} \quad b_{2,3}=-a_{2,3}.$$
	Therefore,
	\[
	B= \left( \begin {array}{ccc} -a_{{1,1}}&-a_{{1,2}}-a_{{2,1}}-b_{{2,1}}&
	-a_{{1,3}}\\ \noalign{\medskip}b_{{2,1}}&-a_{{2,2}}&-a_{{2,3}}
	\\ \noalign{\medskip}-a_{{3,1}}&-a_{{3,2}}&-a_{{3,3}}\end {array}
	\right).
	\]
\end{proof}

\begin{lemma}\label{lema-svt} Consider $Z_{X_-X_+}\in \mathfrak{Z}^{l}_{\mathcal{I}}$. Then the torus is a sliding region or an escape region, and the sliding vector field $Z^s$ defined over the torus is given by
	\[
	\begin{array}{c}
		Z^s(x,y,z)=-\dfrac{1}{2}(a_{2,1}+b_{2,1})y\dfrac{\partial}{\partial x}+\dfrac12(a_{2,1}+b_{2,1})x\dfrac{\partial}{\partial y}.
	\end{array}
	\]
	Moreover, all other trajectories of $Z^s$ are closed.
\end{lemma}

\begin{proof} We have that $X_+$ and $X_-$ are inelastic, thus $X_+h(x,y,z)X_-h(x,y,z)<0$ for all $(x,y,z) \in T^2$, that is, the torus is a sliding region or an escape region. 
	
	We compute $Z^s$ using the equation (\ref{campo_deslizante}) and the information that $X_+$ and $X_-$ are inelastic, that is, 
	\begin{align*}
		Z^s(x,y,z)&\!=\!\dfrac{1}{X_-h(x,y,z)-X_+h(x,y,z)}\! \left(\!X_-h(x,y,z)X_+(x,y,z)\!-\! X_+h(x,y,z)X_-(x,y,z)\!\right)\\ \noalign{\medskip}
		&\!=\!\dfrac{1}{X_-h(x,y,z)+X_-h(x,y,z)}\! \left(\!X_-h(x,y,z)X_+(x,y,z)\!+\! X_-h(x,y,z)X_-(x,y,z)\!\right)\\ \noalign{\medskip}
		&\!=\!\dfrac{1}{2}(X_+(x,y,z)+X_-(x,y,z))\\\noalign{\medskip}
		&\!=\! \dfrac{1}{2} 
		\left( \begin {array}{ccc} 
		0 &-a_{2,1}-b_{2,1} & 0\\ \noalign{\medskip}
		a_{2,1}+b_{2,1}&0& 0\\\noalign{\medskip}
		0&0&0
		\end {array} \right) 
		\left(\begin{array}{c}
			x\\ \noalign{\medskip}
			y\\ \noalign{\medskip}
			z
		\end{array}\right)\\ \noalign{\medskip}
		&\!=\!-\dfrac{1}{2}(a_{2,1}+b_{2,1})y\dfrac{\partial}{\partial x}+\dfrac12(a_{2,1}+b_{2,1})x\dfrac{\partial}{\partial y}.
	\end{align*}
	
	Considering the vector $\eta=(0,0,1)$, we have that 
	$$\left\langle Z^s(x,y,z), \eta\right\rangle=\left\langle \left(-\dfrac{1}{2}(a_{2,1}+b_{2,1})y,\dfrac12(a_{2,1}+b_{2,1})x,0\right),(0,0,1)  \right\rangle=0. $$
	So the trajectories of the sliding vector field $Z^s$ are contained in the planes orthogonal to the vector $\eta$. Furthermore, as they are contained in $T^2$, we conclude that they are closed.
\end{proof}

In Figure \ref{torodeslize} we illustrate the sliding vector field $Z^s$ defined over the torus given by Lemma \ref{lema-svt}.

\begin{figure}[h]
	\centering
	\includegraphics[scale=0.45]{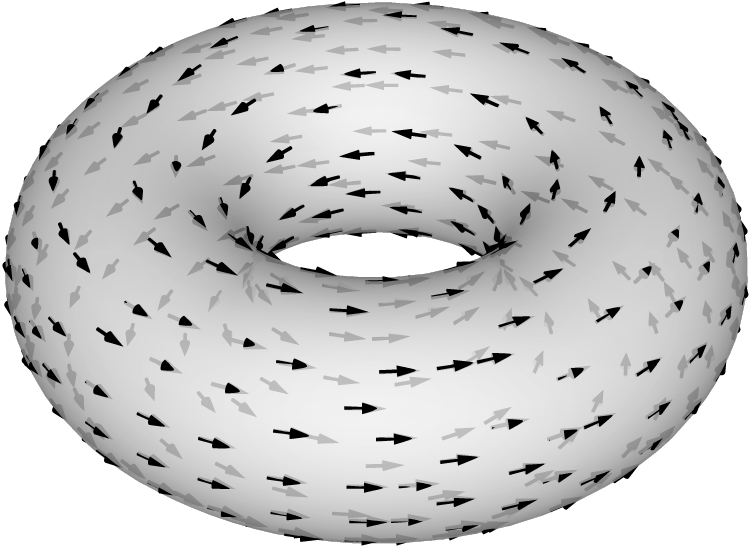}
	\caption{Phase portrait of the sliding vector field $Z^s$.}
	\label{torodeslize}
\end{figure}

Now we can prove Theorem \ref{propfinal1}.

\begin{proof}[Proof of Theorem \ref{propfinal1}]Note that by the final part of the proof of Lemma \ref{lema-svt}, if  $Z_{X_-X_+}\in \mathfrak{Z}^{l}_{\mathcal{I}}$ and $a_{2,1}+b_{2,1}\neq 0$, then the solutions of $Z^s$ are circles, that can be obtained as the intersection of horizontal planes and the torus. This concludes the proof.
\end{proof}

\begin{remark}Note that the results of this section are valid for the torus given by the equation
$$(x^2+y^2+z^2+R^2-r^2)^2=4R^2(x^2+y^2),$$
for $R>0$ and $r>0$. 
\end{remark}
 
\section{Classification of the tangency sets on the torus}\label{sectang}

In this section, we present a classification for the set of tangency points on the torus when $Z_{X_-X_+}\in \mathfrak{Z}^{l}_{\mathcal{I}}$.

We remark that the tangency points of the vector fields $X_+$ and $X_-$ coincide, because $Z_{X_-X_+}$ is inelastic over $T^2$. Hence, it is sufficient to study the tangency either of $X_+$ or $X_-$. In this text, we consider $X_+$.

From Lemma \ref{lema1}, we have that
\[X_+h(x,y,z)=q_2(x,y,z)+q_4(x,y,z),\]
where
$$q_2(x,y,z)=-32[a_{1,1}x^2+a_{2,2}y^2+(a_{1,2}+a_{2,1})xy+a_{1,3}xz+a_{2,3}yz]$$
and
\begin{align*}
	q_4(x,y,z)&=4(x^2+y^2+z^2+3)[a_{1,1}x^2+a_{2,2}y^2+a_{3,3}z^2+(a_{1,2}+a_{2,1})xy\\ 
	&\quad+(a_{1,3}+a_{3,1})xz+(a_{2,3}+a_{3,2})yz]\\ 
	&=4(x^2+y^2+z^2+3)Q_2(x,y,z),
\end{align*}
being
$$Q_2(x,y,z)=(x,y,z)^T  \left( \begin {array}{ccc} 
a_{1,1} &a_{1,2} & a_{1,3}\\ \noalign{\medskip}
a_{2,1} & a_{2,2} & a_{2,3}\\\noalign{\medskip}
a_{3,1} & a_{3,2} & a_{3,3}
\end {array} \right)  (x,y,z). $$

It is hard to determine the zeros of $X_+h$ for general functions $q_2$ and $q_4$, so we discuss some particular choices for $q_4$ and $a_{i,j}$ with $i,j \in \{1,2,3\}$.

\begin{remark}\label{obstang}
We are going to analyze the intersection between the torus $T^2$ and some planes. We have that $T^2=h^{-1}(\{0\})$, thus
$$h(x,y,z)=0 \Rightarrow (x^2+y^2+z^2+3)^2=16(x^2+y^2) \Rightarrow$$
$$ z=\pm \sqrt{4\sqrt{x^2+y^2}-x^2-y^2-3}\; \mbox{or}\; z=\pm \sqrt{-4\sqrt{x^2+y^2}-x^2-y^2-3}.$$
In this way, given the plane $x=\alpha y$, we get
$$z=\pm \sqrt{4\sqrt{(\alpha y)^2+y^2}-(\alpha y)^2-y^2-3} \Rightarrow z^2=4\sqrt{\alpha^2+1}y-(\alpha^2+1)y^2-3$$
$$\Rightarrow z^2+(\alpha^2+1)y^2-4\sqrt{\alpha^2+1}y+4=-3+4 \Rightarrow z^2+(\alpha^2+1)\left(y^2-\frac{4}{\sqrt{\alpha^2+1}}y+\frac{4}{\alpha^2+1}\right)=1 $$
$$\Rightarrow z^2+(\alpha^2+1)\left(y-\frac{2}{\sqrt{\alpha^2+1}} \right)^2=1\Rightarrow \frac{z^2}{\alpha^2+1}+\left(y-\frac{2}{\sqrt{\alpha^2+1}} \right)^2=\frac{1}{\alpha^2+1} $$
and 
$$z=\pm \sqrt{-4\sqrt{(\alpha y)^2+y^2}-(\alpha y)^2-y^2-3} \Rightarrow z^2=-4\sqrt{\alpha^2+1}y-(\alpha^2+1)y^2-3$$
$$\Rightarrow z^2+(\alpha^2+1)y^2+4\sqrt{\alpha^2+1}y+4=-3+4 \Rightarrow z^2+(\alpha^2+1)\left(y^2+\frac{4}{\sqrt{\alpha^2+1}}y+\frac{4}{\alpha^2+1}\right)=1$$
$$\Rightarrow z^2+(\alpha^2+1)\left(y+\frac{2}{\sqrt{\alpha^2+1}} \right)^2=1\Rightarrow \frac{z^2}{\alpha^2+1}+\left(y+\frac{2}{\sqrt{\alpha^2+1}} \right)^2=\frac{1}{\alpha^2+1},$$
that is, the intersection between $T^2$ and $x=\alpha y$ is the union of the ellipses $\mathcal{E}_{\tilde{\alpha}}$ and $\mathcal{E}_{\bar{\alpha}}$, respectively. 
Considering the plane $y=\beta x$, we obtain 
$$z=\pm \sqrt{4\sqrt{x^2+(\beta x)^2}-x^2-(\beta x)^2-3} \Rightarrow \frac{z^2}{\beta^2+1}+\left(x-\frac{2}{\sqrt{\beta^2+1}} \right)^2=\frac{1}{\beta^2+1} $$
and 
$$z=\pm \sqrt{4\sqrt{x^2+(\beta x)^2}-x^2-(\beta x)^2-3} \Rightarrow \frac{z^2}{\beta^2+1}+\left(x+\frac{2}{\sqrt{\beta^2+1}} \right)^2=\frac{1}{\beta^2+1},$$
that is, the intersection between $T^2$ and $y=\beta x$ is the union of the ellipses $\mathcal{E}_{\tilde{\beta}}$ and $\mathcal{E}_{\bar{\beta}}$, respectively. 
Taking the plane $x=0$, we get
$$z=\pm \sqrt{4y-y^2-3} \Rightarrow z^2+y^2-4y+4=-3+4 \Rightarrow z^2+(y-2)^2=1,$$
and
$$z=\pm \sqrt{-4y-y^2-3} \Rightarrow z^2+y^2+4y+4=-3+4 \Rightarrow z^2+(y+2)^2=1,$$
that is,  the intersection between $T^2$ and $x=0$ is the union of the circles $C_{\tilde{x}}$ and $C_{\bar{x}}$, respectively. 
Given the plane $y=0$, we obtain
$$z=\pm \sqrt{4x-x^2-3} \Rightarrow z^2+x^2-4x+4=-3+4 \Rightarrow z^2+(x-2)^2=1,$$
and
$$z=\pm \sqrt{-4x-x^2-3} \Rightarrow z^2+x^2+4x+4=-3+4 \Rightarrow z^2+(x+2)^2=1,$$
that is, the intersection between $T^2$ and $y=0$ is the union of the circles $C_{\tilde{y}}$ and $C_{\bar{y}}$, respectively. Finally, considering the plane $z=0$, we get
$$h(x,y,0)=(x^2+y^2+3)^2-16(x^2+y^2)=0 \Rightarrow (x^2+y^2)^2-10(x^2+y^2)+9=0$$
$$\Rightarrow x^2+y^2=1 \quad \mbox{or} \quad x^2+y^2=9,$$
that is,  the intersection between $T^2$ and $z=0$ is the union of the circles $C_{\tilde{z}}$ and $C_{\bar{z}}$, respectively.
\end{remark}

\begin{remark}\label{obstang2}
	Now, we are going to analyze the intersection between the torus $T^2$ and the sphere $x^2+y^2+z^2=\gamma$. We have that $T^2=h^{-1}(\{0\})$, then
	$$h(x,y,z)=0 \Rightarrow (x^2+y^2+z^2+3)^2=16(x^2+y^2).$$
	Thus, given the sphere $x^2+y^2+z^2=\gamma$, we get
	$$(\gamma+3)^2=16(x^2+y^2) \Rightarrow x^2+y^2=\frac{(\gamma+3)^2}{16},$$
	hence, 
	$$z^2=\gamma-\frac{(\gamma+3)^2}{16} \Rightarrow z^2=\frac{16\gamma-(\gamma^2+6\gamma+9)}{16}$$
	$$\Rightarrow z^2=\frac{-(\gamma^2-10\gamma+9)}{16}\Rightarrow z=\pm\frac{\sqrt{-(\gamma-9)(\gamma-1)}}{4}.$$
	Remember that $z\in \mathbb{R}$, thus $(\gamma-9)(\gamma-1)\leq 0$, that is, $1\leq \gamma \leq 9$. Furthermore, we note that $0\leq \sqrt{-(\gamma-9)(\gamma-1)} \leq 4$, since $1\leq \gamma \leq 9$. In this way, $-1\leq z \leq 1$.
	
	Therefore, if $\gamma \in [1,9]$, the intersection between $T^2$ and $x^2+y^2+z^2=\gamma$ is the union of the circles in the planes 
	$$z=\frac{\sqrt{-(\gamma-9)(\gamma-1)}}{4} \quad \mbox{and} \quad z=-\frac{\sqrt{-(\gamma-9)(\gamma-1)}}{4},$$
	with radius $\frac{\gamma+3}{4}$, denoted by $C_{\tilde{\gamma}}$ and $C_{\bar{\gamma}}$, respectively. Otherwise, the intersection is empty.
\end{remark}

\begin{lemma}\label{q4}Consider $q_4(x,y,z)\equiv 0$. Then the tangency set over $T^2$ is the union of four ellipses.
\end{lemma}

\begin{proof}
	We have $q_4(x,y,z)\equiv 0$, that is,
	$$a_{1,1}x^2+a_{2,2}y^2+a_{3,3}z^2+(a_{1,2}+a_{2,1})xy+(a_{1,3}+a_{3,1})xz+(a_{2,3}+a_{3,2})yz=0.$$
	Hence, we obtain
	$$a_{1,1}=a_{2,2}=a_{3,3}=0, \quad a_{1,2}=-a_{2,1}, \quad a_{1,3}=-a_{3,1}, \quad a_{2,3}=-a_{3,2}.$$
	So,
	$$q_2(x,y,z)=-32(a_{3,1}xz+a_{3,2}yz),$$
	that is, $X_+h$ reduces to \[X_+h(x,y,z)=-32a_{3,1}xz-32a_{3,2}yz,\] which implies that $X_+h(x,y,z)=0$ when 
	$$z=0 \quad \mbox{or} \quad a_{3,1}x+a_{3,2}y=0.$$
	
	We know that the tangency points are of the form $(x,y,z)\in T^2$, such that $X_+h(x,y,z)=0$. Thus, if $z=0$, by Remark \ref{obstang}, we get $C_{\tilde{z}}$ and $C_{\bar{z}}$. Now if $a_{3,1}x+a_{3,2}y=0$, by Remark \ref{obstang}, we get
	\begin{itemize}
		\item the ellipses $\mathcal{E}_{\tilde{\alpha}}$ and $\mathcal{E}_{\bar{\alpha}}$ with $\alpha= -\dfrac{a_{3,2}}{a_{3,1}}$, when $a_{3,1}\neq 0$ and $a_{3,2}\neq 0$;
		\item the ellipses $\mathcal{E}_{\tilde{\alpha}}$ and $\mathcal{E}_{\bar{\alpha}}$, with $\alpha=-a_{3,2}$, when $a_{3,1}=0$ and $a_{3,2}\neq 0$;
		\item the ellipses $\mathcal{E}_{\tilde{\beta}}$ and $\mathcal{E}_{\bar{\beta}}$, with $\beta=-a_{3,1}$, when $a_{3,1}\neq 0$ and $a_{3,2}=0$.
	\end{itemize}
	
	Therefore, we obtain that the set of tangency points is formed by four ellipses, which we illustrate in Figure \ref{ftang1}.
\end{proof}

\begin{lemma}\label{tang1}
	Consider $q_4(x,y,z)=4(x^2+y^2+z^2+3)(a_{1,3}+a_{3,1})xz$ and $a_{3,2}=0$. Then, the tangency set over $T^2$ is the union of four or six ellipses.
\end{lemma}

\begin{proof}
We have $q_4(x,y,z)=4(x^2+y^2+z^2+3)(a_{1,3}+a_{3,1})xz$ and $a_{3,2}=0$, thus
$$a_{1,1}=a_{2,2}=a_{3,3}=0, \quad a_{1,2}=-a_{2,1} \quad \mbox{and} \quad a_{2,3}=-a_{3,2}=0.$$
Then,
$$q_2(x,y,z)=-32a_{1,3}xz,$$
that is, $X_+h$ reduces to
$$ X_+h(x,y,z)=4(x^2+y^2+z^2+3)(a_{1,3}+a_{3,1})xz-32a_{1,3}xz,$$
thus, $X_+h(x,y,z)=0$ when $$xz[-5a_{1,3}+3a_{3,1}+(x^2+y^2+z^2)(a_{1,3}+a_{3,1})]=0.$$
It implies that $$x=0,\quad z=0 \quad \mbox{or} \quad (x^2+y^2+z^2)(a_{1,3}+a_{3,1})=5a_{1,3}-3a_{3,1}.$$

We know that the tangency points are given by $(x,y,z) \in T^2$, such that $X_+h(x,y,z)=0$. Hence, if $x=0$, by Remark \ref{obstang}, we get $C_{\tilde{x}}$ and $C_{\bar{x}}$. 
Now if $z=0$, by Remark \ref{obstang}, we obtain $C_{\tilde{z}}$ and $C_{\bar{z}}$. Finally, if $(x^2+y^2+z^2)(a_{1,3}+a_{3,1})=5a_{1,3}-3a_{3,1}$, by Remark \ref{obstang2},
\begin{itemize}
	\item when $a_{1,3} \neq 0$ and $a_{3,1}\neq 0$, if $\gamma \in [1,9]$, with
	$$\gamma=\frac{5a_{1,3}-3a_{3,1}}{a_{1,3}+a_{3,1}},$$
	we obtain the circles $C_{\tilde{\gamma}}$ and $C_{\bar{\gamma}}$,
	otherwise the intersection is empty;
	\item when $a_{1,3} \neq 0$ and $a_{3,1}= 0$, we get the circles $C_{\tilde{\gamma}}$ and $C_{\bar{\gamma}}$, with $\gamma=5$;
	\item when $a_{1,3} = 0$ and $a_{3,1}\neq 0$, we obtain that the intersection is empty.
\end{itemize} 

Therefore, we get that the tangency set over $T^2$ is formed by four or six ellipses. In figure \ref{ftang2} we illustrate the case where it admits six ellipses. 
\end{proof}	

\begin{figure}[h]
	\begin{subfigure}{0.5\textwidth}
			\centering
			\includegraphics[scale=0.3]{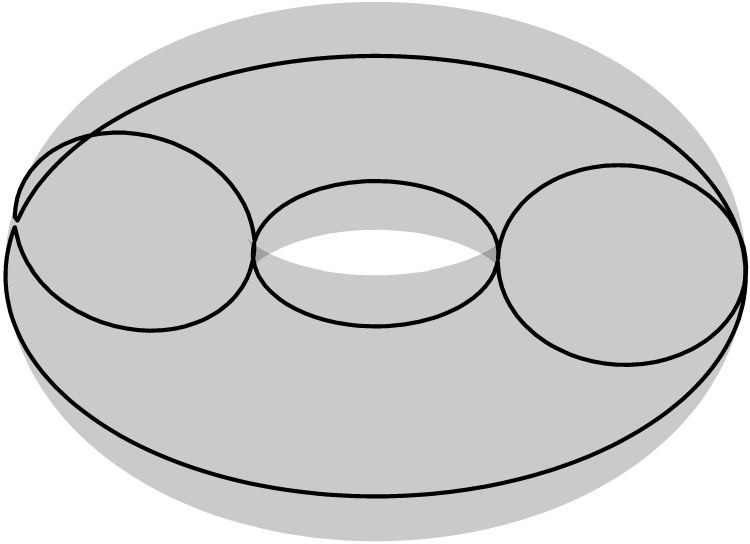}
			\caption{ Four ellipses}
			\label{ftang1}
		\end{subfigure}%
	\begin{subfigure}{0.5\textwidth}
			\centering
			\includegraphics[scale=0.3]{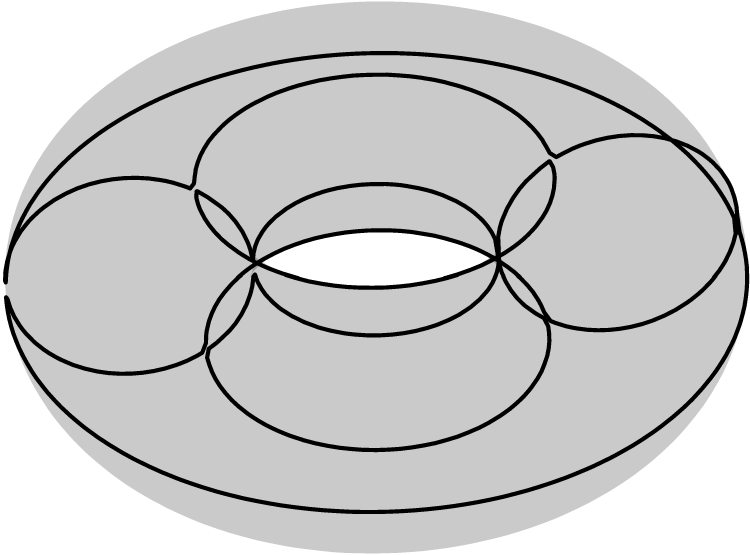}
			\caption{Six ellipses}
			\label{ftang2}
		\end{subfigure}%
	\caption{Tangency set over $T^2$ from Lemma \ref{q4} (a) and Lemma \ref{tang1} (b).}
	\label{lema12}
\end{figure}

\begin{lemma}\label{tang2}
	Consider $q_4(x,y,z)=4(x^2+y^2+z^2+3)(a_{2,3}+a_{3,2})yz$ and $a_{3,1}=0$. Then, the tangency set over $T^2$ is the union of six circles or the union of four circles.
\end{lemma}

\begin{proof}
	We have $q_4(x,y,z)=4(x^2+y^2+z^2+3)(a_{2,3}+a_{3,2})yz$ and $a_{3,1}=0$, thus
	$$a_{1,1}=a_{2,2}=a_{3,3}=0, \quad a_{1,2}=-a_{2,1} \quad \mbox{and} \quad a_{1,3}=-a_{3,1}=0.$$
	Hence,
	$$q_2(x,y,z)=-32a_{2,3}xz,$$
	that is, $X_+h$ reduces to
	$$X_+h(x,y,z)=4(x^2+y^2+z^2+3)(a_{2,3}+a_{3,2})xz-32a_{2,3}xz.$$
	So, $X_+h(x,y,z)=0$ when $$xz[-5a_{2,3}+3a_{3,2}+(x^2+y^2+z^2)(a_{2,3}+a_{3,2})]=0.$$
	It implies that $$x=0,\quad z=0 \quad \mbox{or} \quad (x^2+y^2+z^2)(a_{2,3}+a_{3,2})=5a_{2,3}-3a_{3,2}.$$
	
	We know that the tangency points are of the form $(x,y,z) \in T^2$, such that $X_+h(x,y,z)=0$. Thus, if $x=0$, by Remark \ref{obstang}, we obtain $C_{\tilde{x}}$ and $C_{\bar{x}}$. 
	Now if $z=0$, by Remark \ref{obstang}, we get $C_{\tilde{z}}$ and $C_{\bar{z}}$. Finally, if $(x^2+y^2+z^2)(a_{2,3}+a_{3,2})=5a_{2,3}-3a_{3,2}$, by Remark \ref{obstang2},
	\begin{itemize}
		\item when $a_{2,3} \neq 0$ and $a_{3,2}\neq 0$, if $\gamma \in [1,9]$, with
		$$\gamma=\frac{5a_{2,3}-3a_{3,2}}{a_{2,3}+a_{3,2}},$$
		we obtain the circles $C_{\tilde{\gamma}}$ and $C_{\bar{\gamma}}$,
		otherwise the intersection is empty;
		\item when $a_{2,3} \neq 0$ and $a_{3,2}= 0$, we get the circles $C_{\tilde{\gamma}}$ and $C_{\bar{\gamma}}$, with $\gamma=5$;
		\item when $a_{2,3} = 0$ and $a_{3,2}\neq 0$, we obtain that the intersection is empty.
	\end{itemize} 
	
	Therefore, we get that the tangency set over $T^2$ is formed by four or six ellipses. In figure \ref{ftang3} we illustrate the case where it admits six ellipses.
\end{proof}	

\begin{lemma}\label{tang3}
	Consider $q_4(x,y,z)=4(x^2+y^2+z^2+3)(a_{3,1}xz+a_{3,2}yz)$. Then, the tangency set over $T^2$ is the union of four ellipses.
\end{lemma}

\begin{proof}
	We have $q_4(x,y,z)=4(x^2+y^2+z^2+3)(a_{3,1}xz+a_{3,2}yz)$, thus
	$$a_{1,1}=a_{2,2}=a_{3,3}=a_{1,3}=a_{2,3}=0, \quad \mbox{and} \quad a_{1,2}=-a_{2,1}.$$
	Hence,
	$$q_2(x,y,z)=0,$$
	that is, $X_+h$ reduces to
	$$X_+h(x,y,z)=4(x^2+y^2+z^2+3)(a_{3,1}xz+a_{3,2}yz).$$
	So, $X_+h(x,y,z)=0$ when
	$$z=0 \quad \mbox{or} \quad a_{3,1}x+a_{3,2}y=0.$$
	
	We know that the tangency points are given by $(x,y,z) \in T^2$, such that, $X_+h(x,y,z)=0$. Thus, if $z=0$, by Remark \ref{obstang}, we get $C_{\tilde{z}}$ and $C_{\bar{z}}$. Finally, if $ a_{3,1}x+a_{3,2}y=0$, 
	by Remark \ref{obstang}, we get
	\begin{itemize}
		\item the ellipses $\mathcal{E}_{\tilde{\alpha}}$ and $\mathcal{E}_{\bar{\alpha}}$ with $\alpha= -\dfrac{a_{3,2}}{a_{3,1}}$, when $a_{3,1}\neq 0$ and $a_{3,2}\neq 0$;
		\item the ellipses $\mathcal{E}_{\tilde{\alpha}}$ and $\mathcal{E}_{\bar{\alpha}}$, with $\alpha=-a_{3,2}$, when $a_{3,1}=0$ and $a_{3,2}\neq 0$;
		\item the ellipses $\mathcal{E}_{\tilde{\beta}}$ and $\mathcal{E}_{\bar{\beta}}$, with $\beta=-a_{3,1}$, when $a_{3,1}\neq 0$ and $a_{3,2}=0$.
	\end{itemize}	
	
	Therefore, we get that the set of tangency points is formed by four ellipses, as we illustrate in Figure \ref{ftang4}.
\end{proof}	

We note that, by Lemma \ref{q4} and Lemma \ref{tang3}, we get the same set of tangency points when $q_4(x,y,z)=0$ and $q_4(x,y,z)=4(x^2+y^2+z^2+3)(a_{3,1}xz+a_{3,2}yz)$.

\begin{figure}[h]
	\begin{subfigure}{0.5\textwidth}
		\centering
		\includegraphics[scale=0.3]{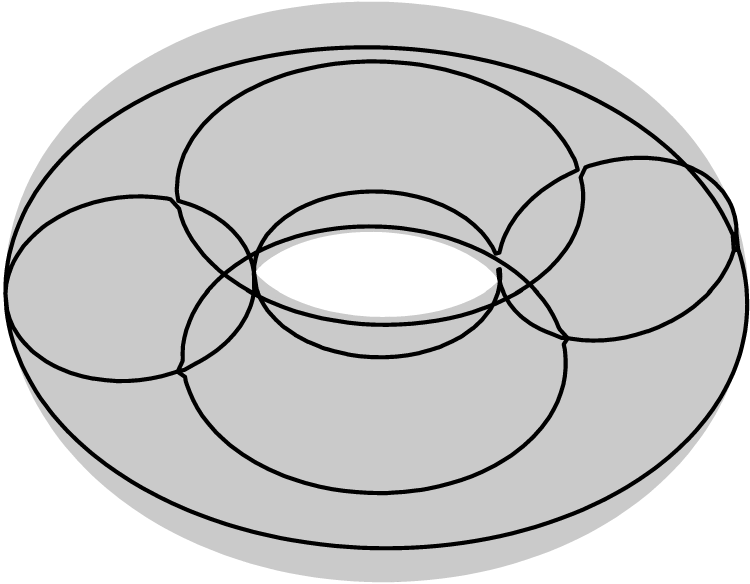}
		\caption{ Six ellipses}
		\label{ftang3}
	\end{subfigure}%
	\begin{subfigure}{0.5\textwidth}
		\centering
		\includegraphics[scale=0.3]{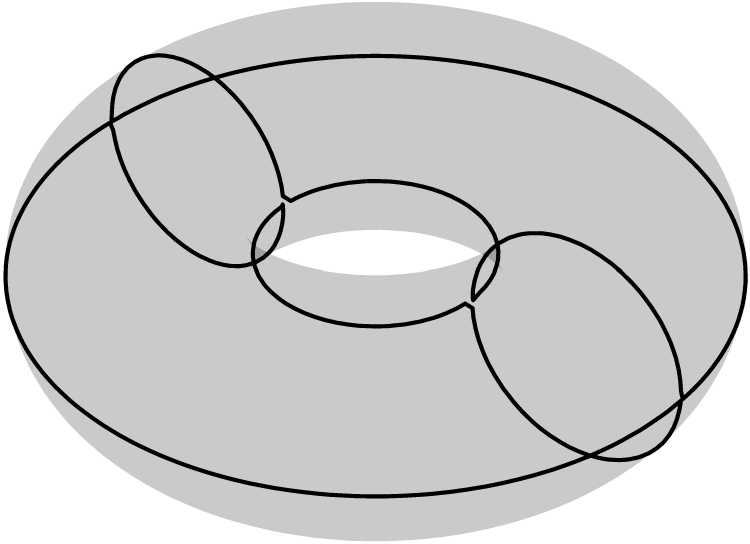}
		\caption{Four ellipses}
		\label{ftang4}
	\end{subfigure}%
	\caption{Tangency set over $T^2$ from Lemma \ref{tang2} (a) and Lemma \ref{tang3} (b).}
\end{figure} 

\begin{lemma}\label{tang4}
	Consider $q_4(x,y,z)=4(x^2+y^2+z^2+3)(a_{1,3}xz+a_{2,3}yz)$. Then, the tangency set over $T^2$ is the union of four or six ellipses.
\end{lemma}

\begin{proof}
	We have $q_4(x,y,z)=4(x^2+y^2+z^2+3)(a_{1,3}xz+a_{2,3}yz)$, thus
	$$a_{1,1}=a_{2,2}=a_{3,3}=a_{3,1}=a_{3,2}=0, \quad \mbox{and} \quad a_{1,2}=-a_{2,1}.$$
	Hence,
	$$q_2(x,y,z)=-32(a_{1,3}xz+a_{2,3}yz),$$
	that is, $X_+h$ reduces to
	$$X_+h(x,y,z)=4(x^2+y^2+z^2-5)(a_{1,3}xz+a_{2,3}yz).$$
	So, $X_+h(x,y,z)=0$ when
	$$x^2+y^2+z^2=5, \quad z=0, \quad \mbox{or} \quad a_{1,3}x+a_{2,3}y=0.$$
	
	We know that the tangency points are of the form $(x,y,z) \in T^2$, such that $X_+h(x,y,z)=0$. Thus, if $z=0$, by Remark \ref{obstang}, we obtain $C_{\tilde{z}}$ and $C_{\bar{z}}$. Now, if $x^2+y^2+z^2=5$, by Remark \ref{obstang2}, we have the circles $C_{\tilde{\gamma}}$ and $C_{\bar{\gamma}}$, with $\gamma=5$, that is, $x^2+y^2=4$ in the planes $z=1$ and $z=-1$. Finally, if $a_{1,3}x+a_{2,3}y=0$, 
	by Remark \ref{obstang}, we obtain
	\begin{itemize}
		\item the ellipses $\mathcal{E}_{\tilde{\alpha}}$ and $\mathcal{E}_{\bar{\alpha}}$ with $\alpha= -\dfrac{a_{2,3}}{a_{1,3}}$, when $a_{1,3}\neq 0$ and $a_{2,2}\neq 0$;
		\item the ellipses $\mathcal{E}_{\tilde{\alpha}}$ and $\mathcal{E}_{\bar{\alpha}}$, with $\alpha=-a_{2,3}$, when $a_{1,3}=0$ and $a_{2,3}\neq 0$;
		\item the ellipses $\mathcal{E}_{\tilde{\beta}}$ and $\mathcal{E}_{\bar{\beta}}$, with $\beta=-a_{1,3}$, when $a_{1,3}\neq 0$ and $a_{2,3}=0$.
	\end{itemize}	
	
	Therefore, we have that the set of tangency points is formed by four or six ellipses. In Figure \ref{ftang5} we illustrate the case where it admits six ellipses. 
\end{proof}

\begin{lemma}\label{tang5}
	Consider $q_4(x,y,z)=4a_{3,3}z^2(x^2+y^2+z^2+3)$ and $a_{3,1}=a_{3,2}=0$. Then, the tangency set over $T^2$ is the union of two circles.
\end{lemma}

\begin{proof}
	We have $q_4(x,y,z)=4a_{3,3}z^2(x^2+y^2+z^2+3)$ and $a_{3,1}=a_{3,2}=0$, then
	$$a_{1,1}=a_{2,2}=0,\quad a_{1,3}=-a_{3,1}=0, \quad a_{2,3}=-a_{3,2}=0 \quad \mbox{and} \quad a_{1,2}=-a_{2,1}.$$
	Hence,
	$$q_2(x,y,z)=0,$$
	that is, $X_+h$ reduces to
	$$X_+h(x,y,z)=4a_{3,3}z^2(x^2+y^2+z^2+3).$$
	So, $X_+h(x,y,z)=0$ when $z=0$. 
	
	We know that the tangency points are given by $(x,y,z) \in T^2$, such that $X_+h(x,y,z)=0$. Thus, when $z=0$, by Remark \ref{obstang}, we obtain $C_{\tilde{z}}$ and $C_{\bar{z}}$.
	
	Therefore, we have that the set of tangency points is formed by two circles, which we illustrate in Figure \ref{ftang6}.
\end{proof}

\begin{lemma}\label{tang6}
	Consider $q_4(x,y,z)=4(x^2+y^2+z^2+3)[a_{1,1}x^2+a_{2,2}y^2+(a_{1,2}+a_{2,1})xy]$ and $a_{1,3}=a_{2,3}=0$. Then the tangency set over $T^2$ is the union of two, four or six ellipses.
\end{lemma}

\begin{proof}
	We have $4(x^2+y^2+z^2+3)[a_{1,1}x^2+a_{2,2}y^2+(a_{1,2}+a_{2,1})xy]$ and $a_{1,3}=a_{2,3}=0$, thus
	$$a_{3,3}=0, \quad a_{1,3}=-a_{3,1}=0 \quad \mbox{and} \quad a_{2,3}=-a_{3,2}=0.$$
	Hence,
	$$q_2(x,y,z)=-32[a_{1,1}x^2+a_{2,2}y^2+(a_{1,2}+a_{2,1})xy],$$
	that is, $X_+h$ reduces to
	$$X_+h(x,y,z)=4(x^2+y^2+z^2-5)[a_{1,1}x^2+a_{2,2}y^2+(a_{1,2}+a_{2,1})xy].$$
	So, $X_+h(x,y,z)=0$ when
	$$x^2+y^2+z^2=5 \quad \mbox{or} \quad a_{1,1}x^2+a_{2,2}y^2+(a_{1,2}+a_{2,1})xy=0.$$
	
	We know that the tangency points are of the form $(x,y,z) \in T^2$, such that $X_+h(x,y,z)=0$. Thus, if $x^2+y^2+z^2=5$, by Remark \ref{obstang2}, we have the circles $C_{\tilde{\gamma}}$ and $C_{\bar{\gamma}}$, with $\gamma=5$, that is, $x^2+y^2=4$ in the planes $z=1$ and $z=-1$. Finally, if $a_{1,1}x^2+a_{2,2}y^2+(a_{1,2}+a_{2,1})xy=0$, 
	by Remark \ref{obstang}, we obtain
	\begin{itemize}
		\item the ellipses $\mathcal{E}_{\tilde{\alpha}_+}$, $\mathcal{E}_{\bar{\alpha}_+}$, $\mathcal{E}_{\tilde{\alpha}_-}$ and $\mathcal{E}_{\bar{\alpha}_-}$, with 
		$$\alpha_{\pm}= \frac{-a_{2,1}-a_{1,2}\pm \sqrt{(a_{2,1}+a_{1,2})^2-4a_{1,1}a_{2,2}}}{2a_{1,1}},$$
		when $a_{i,j}\neq 0$ for $i,j\in \{1,2\}$ and $(a_{2,1}+a_{1,2})^2-4a_{1,1}a_{2,2}>0$;
		\item the ellipses $\mathcal{E}_{\tilde{\alpha}_+}$, $\mathcal{E}_{\bar{\alpha}_+}$, $\mathcal{E}_{\tilde{\alpha}_-}$ and $\mathcal{E}_{\bar{\alpha}_-}$, with 
		$$\alpha_{\pm}= \frac{-a_{1,2}\pm \sqrt{(a_{1,2})^2-4a_{1,1}a_{2,2}}}{2a_{1,1}},$$
		when $a_{2,1}=0$, $a_{i,j}\neq 0$ for $i,j\in \{1,2\}$ and $(a_{1,2})^2-4a_{1,1}a_{2,2}>0$;
		\item the ellipses $\mathcal{E}_{\tilde{\alpha}_+}$, $\mathcal{E}_{\bar{\alpha}_+}$, $\mathcal{E}_{\tilde{\alpha}_-}$ and $\mathcal{E}_{\bar{\alpha}_-}$, with 
		$$\alpha_{\pm}= \frac{-a_{2,1}\pm \sqrt{(a_{2,1})^2-4a_{1,1}a_{2,2}}}{2a_{1,1}},$$
		when $a_{1,2}=0$, $a_{i,j}\neq 0$ for $i,j\in \{1,2\}$ and $(a_{2,1})^2-4a_{1,1}a_{2,2}>0$;
		\item the ellipses $\mathcal{E}_{\tilde{\alpha}}$ and  $\mathcal{E}_{\bar{\alpha}}$, with 
		$$\alpha= -\frac{a_{2,1}+a_{1,2}}{a_{1,1}},$$
		and the circles $C_{\tilde{x}}$ and $C_{\bar{x}}$,
		when $a_{2,2}=0$ and $a_{i,j}\neq 0$ for $i,j\in \{1,2\}$;
		\item the ellipses $\mathcal{E}_{\tilde{\beta}}$ and  $\mathcal{E}_{\bar{\beta}}$, with 
		$$\beta= -\frac{a_{2,1}+a_{1,2}}{a_{2,2}},$$
		and the circles $C_{\tilde{y}}$ and $C_{\bar{y}}$,
		when $a_{1,1}=0$ and $a_{i,j}\neq 0$ for $i,j\in \{1,2\}$;
		\item the ellipses $\mathcal{E}_{\tilde{\alpha}_+}$, $\mathcal{E}_{\bar{\alpha}_+}$, $\mathcal{E}_{\tilde{\alpha}_-}$ and $\mathcal{E}_{\bar{\alpha}_-}$, with 
		$$\alpha_{\pm}= \pm\sqrt{-\frac{a_{2,2}}{a_{1,1}}},$$
		when $a_{1,2}=a_{2,1}=0$, $a_{1,1}\neq 0$, $a_{2,2}\neq 0$ and $\dfrac{a_{2,2}}{a_{1,1}}<0$;
		\item the ellipses $\mathcal{E}_{\tilde{\alpha}}$ and  $\mathcal{E}_{\bar{\alpha}}$, with 
		$\alpha= -\dfrac{a_{2,1}}{a_{1,1}},$
		and the circles $C_{\tilde{x}}$ and $C_{\bar{x}}$,
		when $a_{2,2}=a_{1,2}=0$, $a_{1,1}\neq 0$ and $a_{2,1}\neq 0$;
		\item the ellipses $\mathcal{E}_{\tilde{\alpha}}$ and  $\mathcal{E}_{\bar{\alpha}}$, with 
		$\alpha= -\dfrac{a_{1,2}}{a_{1,1}},$
		and the circles $C_{\tilde{x}}$ and $C_{\bar{x}}$,
		when $a_{2,2}=a_{2,1}=0$, $a_{1,1}\neq 0$ and $a_{1,2}\neq 0$;
		\item the ellipses $\mathcal{E}_{\tilde{\beta}}$ and  $\mathcal{E}_{\bar{\beta}}$, with 
		$\beta= -\dfrac{a_{2,1}}{a_{2,2}},$
		and the circles $C_{\tilde{y}}$ and $C_{\bar{y}}$,
		when $a_{1,1}=a_{1,2}=0$, $a_{2,2}\neq 0$ and $a_{2,1}\neq 0$;
		\item the ellipses $\mathcal{E}_{\tilde{\beta}}$ and  $\mathcal{E}_{\bar{\beta}}$, with 
		$\beta= -\dfrac{a_{1,2}}{a_{2,2}},$
		and the circles $C_{\tilde{y}}$ and $C_{\bar{y}}$,
		when $a_{1,1}=a_{2,1}=0$, $a_{2,2}\neq 0$ and $a_{1,2}\neq 0$;
		\item the circles $C_{\tilde{x}}$, $C_{\bar{x}}$, $C_{\tilde{y}}$ and $C_{\bar{y}}$,
		when $a_{1,1}=a_{2,2}=0$ and $a_{1,2}$ or $a_{2,1}$ is nonzero;
		\item the circles $C_{\tilde{x}}$ and $C_{\bar{x}}$,
		when $a_{2,2}=a_{1,2}=a_{2,1}=0$ and $a_{1,1}\neq 0$;
		\item the circles $C_{\tilde{y}}$ and $C_{\bar{y}}$,
		when $a_{1,1}=a_{1,2}=a_{2,1}=0$ and $a_{2,2}\neq 0$.
	\end{itemize}	
	
	Therefore, we have that the set of tangency points is formed by two, four or six ellipses. In Figure \ref{ftang7} we illustrate the case where it admits six ellipses.
\end{proof}

\begin{figure}[h]
	\begin{subfigure}{0.33\textwidth}
			\centering
			\includegraphics[scale=0.3]{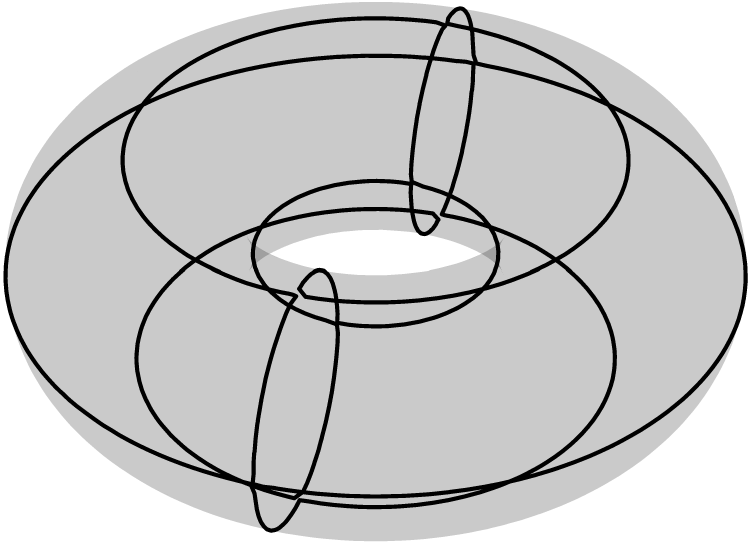}
			\caption{Six ellipes}
			\label{ftang5}
		\end{subfigure}%
	\begin{subfigure}{0.33\textwidth}
			\centering
			\includegraphics[scale=0.3]{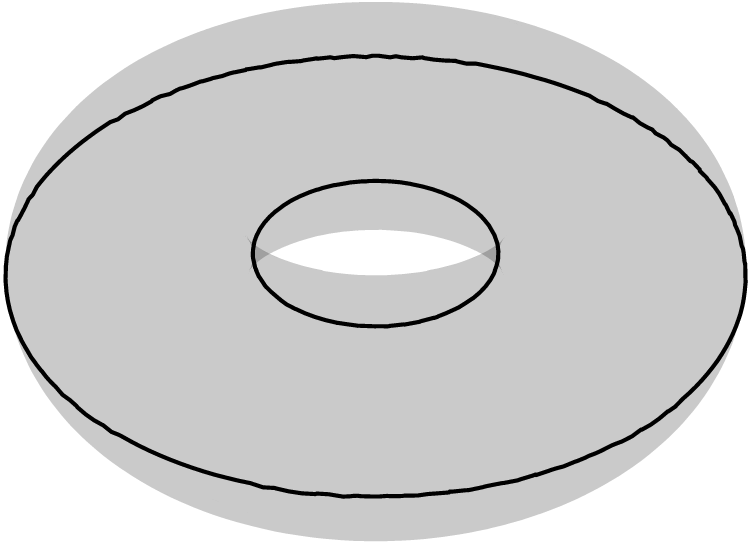}
			\caption{Two circles}
			\label{ftang6}
		\end{subfigure}%
	\begin{subfigure}{0.33\textwidth}
			\centering
			\includegraphics[scale=0.3]{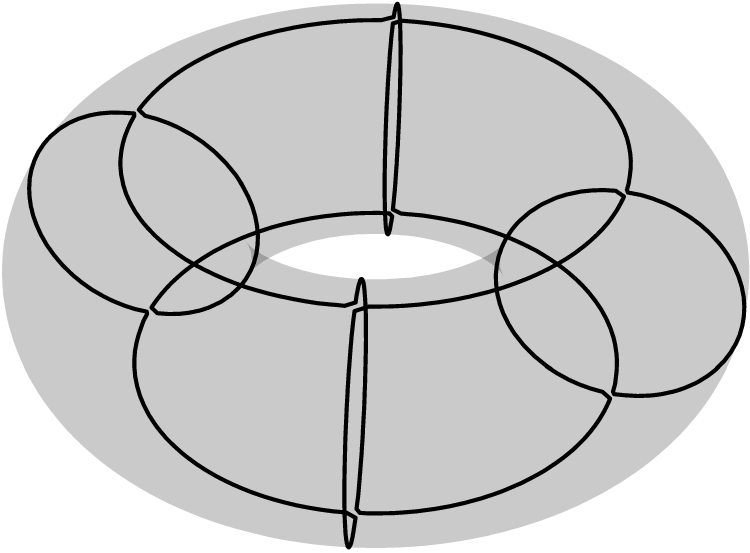}
			\caption{Six ellipses}
			\label{ftang7}
		\end{subfigure}%
	\caption{Tangency set over $T^2$ from Lemma \ref{tang4} (a), Lemma \ref{tang5} (b) and Lemma \ref{tang6} (c).}
	\label{lema567}
\end{figure}

\section{Proof of Theorem \ref{44r4r4r4}} \label{proof2}

From the classification of the tangency sets on the torus made in the previous section, we will show Theorem \ref{44r4r4r4}.



\begin{proof}[Proof of Theorem \ref{44r4r4r4}]
	Consider the set $\mathfrak z$ formed by the vector fields $Z_{X_-X_+}$ that satisfy the conditions of any of the Lemmas \ref{tang1} to \ref{tang6} and {such that the singularities of $X_-,X_+$ are hyperbolic.} Let $Z_{X_-X_+}\in \mathfrak z$. Then, the tangency sets of $Z_{X_-X_+}$ are always formed by a finite number of ellipses, which intersect the trajectories of the sliding vector field in a finite number of points or are exactly equal to the solution.    
	
	Now, consider $\tilde{Z}_{\tilde{X}_-\tilde{X}_+}$ in a small neighborhood of $Z_{X_-X_+}$ in $\mathfrak z$. We note that, as $\tilde{Z}_{\tilde{X}_-\tilde{X}_+}\in \mathfrak z$, this vector field is also inelastic over the torus $T^2$. By the classification given by the Lemmas \ref{tang1} to \ref{tang6}, this vector field admits exactly the same characteristics of $Z_{X_-X_+}$, that is, has all the periodic solutions obtained as the intersection between the horizontal planes and the torus. Therefore, the trajectories are the same and the vector fields are equivalents, with the equivalence defined based on the intersection of the torus with horizontal planes.
\end{proof}

%
%
%
%
%
%

\begin{remark}It is possible to show the previous result considering the torus given by the equation
$$(x^2+y^2+z^2+R^2-r^2)^2=4R^2(x^2+y^2),$$
for $R>0$ and $r>0$. 
\end{remark}

\section{Conclusions}

We proved that under generic conditions the trajectories of $Z_{X_-X_+}\in\mathfrak{Z}^{l}_{\mathcal{I}}$ over the torus are closed curves. Furthermore, we show the existence of a set $\mathfrak{z}\subset\mathfrak{Z}^{l}_{\mathcal{I}}$ such that if $Z_{X_-X_+}\in \mathfrak{z}$ and $\tilde Z_{\tilde{X}_-\tilde{X}_+}\in \mathfrak{z}$ is a vector field in a small neighborhood of $Z_{X_-X_+}\in \mathfrak{z}$, then the respective sliding vector fields $Z^s$ and $\tilde{Z}^s$, defined on the torus, are topologically equivalent.

Remember that the construction of the set $\mathfrak{z}\subset\mathfrak{Z}^{l}_{\mathcal{I}}$ was done from the classification of the tangency sets on the torus, however cases such as
$q_4(x,y,z)=4(x^2+y^2+z^2+3)(a_{3,3}z^2+a_{3,1}xz+a_{3,2}yz)$, where we get $X_+h(x,y,z)=0$ when 
	$$z=0\quad \mbox{or}\quad a_{3,3}z+a_{3,1}x+a_{3,2}y=0,$$
	and $q_4(x,y,z)=4(x^2+y^2+z^2+3)[a_{1,1}x^2+a_{2,2}y^2+(a_{1,2}+a_{2,1})xy+a_{1,3}xz+a_{2,3}yz]$, where we obtain 
	$X_+h(x,y,z)=0$ when 
	$$x^2+y^2+z^3=5\quad \mbox{or}\quad a_{1,1}x^2+a_{2,2}y^2+(a_{1,2}+a_{2,1})xy+a_{1,3}xz+a_{2,3}yz=0,$$
are very complicated due to the possibility to obtain toric sections different from those obtained in the Section \ref{sectang}, when intersecting $T^2$ with $X_+h(x,y,z)=0$, such as the Villarceau's circles. These cases were not considered in the present study and will be treated in a forthcoming paper. See the reference \cite{toric2} for more details.


\section*{Acknowledgements}

R. M. Martins was partially supported by FAPESP grants 2021/08031-9 and 2018/03338-6, CNPq grants 315925/2021-3, 434599/2018-2 and 306287/2024-2. M. D. A. Caldas was partially supported by Coordenação de Aperfeiçoamento de Pessoal de Nível Superior - Brasil (CAPES), grant number 88887.946674/2024-00. This study was financed in part by the Coordenação de Aperfeiçoamento de Pessoal de Nível Superior - Brasil (CAPES) - Finance Code 001.

\bibliographystyle{plain}
\bibliography{bibli}

\begin{thebibliography}{10}

\bibitem{Bro}
B.~Brogliato.
\newblock {\em Nonsmooth Mechanics: Models, Dynamics and Control}.
\newblock Springer, Switzerland, 3 edition, 2016.

\bibitem{ap1}
D.~Chillingworth.
\newblock Discontinuity geometry for an impact oscillator.
\newblock {\em Dynamical Systems}, 17:389--420, 2011.

\bibitem{ap4}
A.~Colombo, M.~di~Bernardo, E.~Fossas, and M.~R. Jeffrey.
\newblock Teixeira singularities in 3d switched feedback control systems.
\newblock {\em Systems and Control Letters}, 59(10):615--622, 2010.

\bibitem{ap2}
E.~Conte, A.~Federici, and J.~P. Zbilut.
\newblock On a simple case of possible non-deterministic chaotic behavior in
  compartment theory of biological observables.
\newblock {\em Chaos Solitons and Fractals}, 22:277--284, 2004.

\bibitem{ap5}
M.~di~Bernardo, C.~J. Budd, A.~R. Champneys, and R.~Kowalczyk.
\newblock {\em Piecewise-Smooth Dynamical Systems – Theory and Applications}.
\newblock Springer-Verlag, 2008.

\bibitem{ap3}
M.~di~Bernardo, A.~Colombo, and E.~Fossas.
\newblock Two-fold singularity in nonsmooth electrical systems.
\newblock In {\em Proceedings of IEEE International Symposium on Circuits and
  Systems}, pages 2713--2716, 2011.

\bibitem{f1}
A.~F. Filippov.
\newblock {\em Differential equations with discontinuous righthand sides},
  volume~18 of {\em Mathematics and its Applications (Soviet Series)}.
\newblock Kluwer Academic Publishers Group, Dordrecht, 1988.

\bibitem{GST}
M.~Guardia, T.~M. Seara, and M.~A. Teixeira.
\newblock Generic bifurcations of low codimension of planar filippov systems.
\newblock {\em Journal of Differential Equations}, 250(4):1967--2023, 2011.

\bibitem{r1}
J.~Llibre, R.~M. Martins, and D.~J. Tonon.
\newblock Limit cycles of piecewise smooth differential equations on two
  dimensional torus.
\newblock {\em Journal of Dynamical and Differential Equations}, 30:1011--1027,
  2018.

\bibitem{survey}
O.~Makarenkov and J.~S.~W. Lamb.
\newblock Dynamics and bifurcations of non smooth systems: A survey.
\newblock {\em Physica D: Nonlinear Phenomena}, 241:1826--1844, 2012.

\bibitem{v2}
M.~Manzatto, D.~D. Novaes, and R.~M. Martins.
\newblock A note on vishik’s normal form.
\newblock {\em Journal of Differential Equations}, 281:442--458, 2021.

\bibitem{r2}
R.~M. Martins and D.~J. Tonon.
\newblock The chaotic behavior of piecewise smooth differential equations on
  two dimensional torus and sphere.
\newblock {\em Dynamical Systems: An International Journal}, 34:356--373, 2019.

\bibitem{marco2}
M.~A. Teixeira.
\newblock Stability conditions for discontinuous vector fields.
\newblock {\em Journal of Differential Equations}, 88:15--29, 1990.

\bibitem{marco3}
M.~A. Teixeira.
\newblock Generic bifurcations of sliding vector fields.
\newblock {\em Journal of Mathematical Analysis and Applications},
  176:436--457, 1993.

\bibitem{marco4}
M.~A. Teixeira.
\newblock Perturbation theory for non-smooth systems.
\newblock In {\em Encyclopedia of Complexity and Systems Science}, page 152.
  2008.

\bibitem{t1}
V.~I. Tkachenko.
\newblock On the existence of a piecewise-smooth invariant torus of an impulse
  system.
\newblock In {\em Methods for investigating differential and
  functional-differential equations}, pages 91--96. Akad. Nauk Ukrain. SSR,
  Inst. Mat., Kiev, 1990.

\bibitem{VO}
P.~L. Varkonyi and Y.~Or.
\newblock Lyapunov stability of a rigid body with two frictional contacts.
\newblock {\em Nonlinear Dynamics}, 88:363--393, 2017.

\bibitem{villanueva2022global}
Y.~Villanueva.
\newblock {\em Global dynamics of inelastic and center-type piecewise smooth
  vector fields in $\mathbb{R}^2$ and $\mathbb{R}^3$}.
\newblock PhD thesis, Universidade Federal de Goiás, 2022.
\newblock Ph.D. thesis in Mathematics.

\bibitem{v1}
S.~M. Vishik.
\newblock Vector fields near the boundary of a manifold.
\newblock {\em Vestnik Moskov Univ. Ser. I, Mat. Meh.}, 27(1):21--28, 1972.

\bibitem{toric2}
T.~R. Werner.
\newblock {\em Rational families of circles and bicircular quartics}.
\newblock PhD thesis, Friedrich-Alexander-Universitaet Erlangen-Nuernberg
  (Germany), 2012.

\end{thebibliography}

\end{document}